\documentclass{IEEEtran}
\usepackage{cite}
\usepackage{amsmath,amssymb,amsfonts}
\usepackage{graphicx}
\usepackage{textcomp}
\usepackage{color}
\usepackage{epsfig}
\usepackage{subfig}
\usepackage{mathrsfs}
\usepackage{booktabs}
\usepackage{epstopdf}
\usepackage{multirow}
\usepackage{amsthm}
\usepackage{hyperref}
\usepackage{fancyhdr}
\usepackage{tabularx}
\usepackage{xfrac}
\usepackage{mathptmx}
\usepackage{mathrsfs}
\usepackage{graphics}
\usepackage{times}

\usepackage[noend]{algpseudocode}
\usepackage{algorithmicx,algorithm}
\usepackage{makecell}

\usepackage[comma,numbers,square,sort&compress]{natbib}
\usepackage{balance}
\usepackage[noend]{algpseudocode}

\newtheorem{theorem}{Theorem}[section]

\newtheorem{defn}{Definition}[section]
\newtheorem{coro}{Corollary}[section]

\newtheorem{lemma}{Lemma}[section]
\newtheorem{remark}{Remark}[section]

\def\BibTeX{{\rm B\kern-.05em{\sc i\kern-.025em b}\kern-.08em
    T\kern-.1667em\lower.7ex\hbox{E}\kern-.125emX}}

\begin{document}

\title{Multi-Agent Coverage Control on Surfaces Using Conformal Mapping}

\author{Chao Zhai, Yuming Wu \thanks{Chao Zhai and Yuming Wu are with School of Automation, China University of Geosciences, Wuhan 430074, China, and with Hubei Key Laboratory of Advanced Control and Intelligent Automation for Complex Systems and also with Engineering Research Center of Intelligent Technology for Geo-Exploration, Ministry of Education. (Email: zhaichao@cug.edu.cn)}
}

\maketitle

\begin{abstract}
Real-time environmental monitoring using a multi-agent system (MAS) has long been a focal point of cooperative control. It is still a challenging task to provide cost-effective services for potential emergencies in surface environments. This paper explores the transformation of a general surface into a two-dimensional (2D) disk through the construction of a conformal mapping. Multiple agents are strategically deployed within the mapped convex disk, followed by mapping back to the original surface environment. This approach circumvents the complexities associated with handling the difficulties and intricacies of path planning. Technical analysis encompasses the design of distributed control laws and the method to eliminate distortions introduced by the mapping. Moreover, the developed coverage algorithm is applied to a scenario of monitoring surface deformation. Finally, the effectiveness of the proposed algorithm is validated through numerical simulations.
\end{abstract}
\begin{IEEEkeywords}
Coverage control, multi-agent systems, conformal mapping.
\end{IEEEkeywords}

\section{Introduction}
\IEEEPARstart{W}{ith} the rapid advancement of robotics techniques, the integration of intelligent sensors and mobile robots is pervading diverse sectors, including wilderness search and rescue \cite{macwan2014multirobot}, border patrols \cite{farinelli2017distributed}, area coverage \cite{cortes2004coverage}, and environmental monitoring \cite{cassandras2005sensor}. In comparison to individual robots, the collaborative execution of tasks by multiple agents not only enhances efficiency but also enhances overall performance. Multi-agent systems seek to delve into the mechanisms and principles governing the coordination and collective problem-solving abilities \cite{zhai2021cooperative}.

Multi-agent coverage control has attracted more and more attentions in recent years. The central challenge in achieving effective coverage of a given environment with multi-agent system lies in formulating control laws that are tailored to the specific characteristics of the environment. These control laws are designed to expedite coverage tasks while ensuring a high standard of coverage quality. Addressing the former concern often involves considering region partitioning schemes, among which common methods include grid partition \cite{kim2016guidance}, equal task division \cite{zhai2023sectorial}, and Voronoi partition \cite{alitappeh2016distributed}. Grid partition, known for its regularity, divides a 2D plane into uniformly shaped grid cells. It can also be extended to three-dimensional (3D) space using cube grids \cite{ecker2021synthetic}. However, grid partition may result in resource wastage in regions with uneven environmental density. Equal task partition aims to balance workloads among regions and is often used in scenarios like sweep coverage \cite{zhai2013decentralized}. However, its adaptability to dynamic environments is limited. Voronoi partition, widely utilized in research, has been integrated into coverage optimization with region constraints. Nevertheless, traditional methods are largely confined to 2D convex environments \cite{jiang2019higher}. For non-convex regions, some studies have introduced geodesic distance as a new metric, accommodating non-convex nature \cite{thanou2013distributed}\cite{wu2023multi}. In higher-dimensional spaces, research has explored 3D Voronoi partition algorithms considering virtual forces \cite{dang2019target}\cite{wang2023cooperative}, partially addressing 3D Voronoi partition. However, direct application in complex 3D environments may lead to the loss of environmental features. As a result, scholars have extended coverage to hybrid environments with mixed dimensions using geodesic Voronoi partitioning, though typically limited to the fusion of one and two dimensions \cite{liu2021multi}. While these partition methods offer diverse insights into multi-agent coverage, each has limitations, especially regarding adaptability to environmental characteristics and the complexity of higher-dimensional spaces.
% Grid partitioning, known for its regularity, is frequently employed to subdivide a two-dimensional plane into uniformly shaped grid cells, facilitating efficient management. This approach can also be extended to three-dimensional space using cube grids \cite{ecker2021synthetic}. Equal task division, on the other hand, is primarily utilized to achieve a balanced distribution of workload, finding practical application in scenarios such as cleaning coverage \cite{zhai2013decentralized}. Voronoi partitioning has been extensively integrated into existing research. Many scholars have tackled static coverage optimization problems with region constraints using this method. However, traditional Voronoi partitioning methods encounter limitations, particularly in two-dimensional convex environments \cite{jiang2019higher}. For high-dimensional spaces, some studies have explored three-dimensional Voronoi partitioning algorithms that take into account virtual forces \cite{dang2019target}\cite{wang2023cooperative}. For non-convex regions, previous research has introduced geodesic distance as a new distance metric instead of Euclidean distance to accommodate the characteristics of non-convex environments \cite{thanou2013distributed}\cite{wu2023multi}. Additionally, scholars have extended the notion to mixed dimensions and hybrid environments using geodesic Voronoi partitioning \cite{liu2021multi}.

For the issue of ensuring coverage quality, in convex regions, it is common to optimize the deployment of agents by solving for the centroids of Voronoi subcells to minimize the cost function. However, in non-convex regions, due to the non-convex nature of the environment, Voronoi subcells may no longer possess convex properties, and the centroids obtained may not exist within the environment. Some studies have addressed this issue by considering the effectiveness of centroids and projecting ineffective centroids onto the boundaries of obstacles, albeit leading to additional path planning \cite{breitenmoser2010voronoi}. To address this issue, some scholars have proposed a mapping approach to transform non-convex regions into convex ones, allowing iterative optimization of the cost function within the mapped region \cite{caicedo2008performing}. Such mapping techniques have also found application in medical imaging \cite{wang2011brain}. However, multi-agent coverage control based on mapping mostly remains in 2D non-convex environments. Moreover, the mapping inevitably introduces region distortions, and the biases caused by distortions accumulate over time, resulting in the performance degradation of coverage algorithms. This occurs because agents may allocate tasks and optimize deployments based on the distorted maping. Furthermore, the mapped environment may lead to the loss of shapes of some regions, making them appear larger or smaller, which affects the balance of task allocation, consequently impairing the quality and efficiency of coverage.

The mapping based coverage methods have demonstrated effectiveness in certain scenarios. However, they still encounter challenges when applied to coverage monitoring on general surfaces. First of all, these methods are primarily confined to 2D planes. When extended to higher dimensions or non-convex environments, they incur significant computational costs, making them impractical for direct application to general surface environments. Secondly, when tackling surface coverage problems, they fail to ensure theoretical optimality and provide inadequate evaluation metrics for coverage performance. Last but not least, complex computational procedures are necessary during partition, and during the search for optimal deployment, points that may not exist within the environment may be identified, requiring additional path planning.

As a remedy, this paper endeavors to develop a distributed coverage control algorithm specifically crafted for monitoring surface environments and the optimization of coverage performance. In essence, the core contributions can be summarized as follows:
\begin{enumerate}
\item By constructing a conformal homeomorphic mapping, we achieve the transformation of a surface environment into a 2D disk. This transformation effectively converts a 2D manifold into a 2D disk, simplifying the environment to a flat, 2D plane. Notably, given the bijective nature of the constructed mapping, coverage control operations executed on the 2D plane can be correspondingly mapped back to the original environment.
\item Propose coverage performance metrics tailored for surfaces for a more realistic assessment of coverage quality. In the case of surface deformation, the approach involves differencing the disk data before and after the surface deformation. This difference yields deformation metrics, facilitating the refinement of coverage performance metrics.
\item On the transformed 2D disk, distributed control laws are designed for agents, propelling each agent towards the local optima within its subregion. The optimized deployment can then be mapped back to the original environment, enhancing global coverage performance. Simultaneously, the paths of agents on the disk are mapped back to the curved surface, eliminating the need for path planning on the original surface.
\end{enumerate}
The rest of this paper is organized as follows. Section~\ref{sec:pre} presents some preliminaries, Section~\ref{sec:hol} describes the conformal mapping algorithm and main results, Section~\ref{sec:sim} gives a demonstration of simulations, and Section~\ref{sec:con} summarizes the paper and discusses future work.

\textit{Notation :} Let $V_i$ denote the Voronoi cell, and ${H}$  is the performance function. In addition, $\left\| \cdot \right\|$ denotes Euclidean norm, and $u_i$ represents control inputs. $d_{\omega}$ characterizes the geodesic distance, and $\Phi (q)$ refers to the density value at point $q$. $\bigtriangleup h$ denotes the maximum slope difference, and $\partial Q$ represents the boundary of region $Q$. Moreover, $\mathbb{D}$ represents the unit disk, and ${E}$ represents the harmonic energy. $\mu$ is the Beltrami coefficient, and $k_{uv}$ represents the weight coefficient of edges. $C_{V_i}$ is the centroid of $V_i$, and $M_{V_i}$ is the mass of $V_i$. Further, $\mathbb{H}$ denotes the upper half of the complex plane.

\section{Preliminaries}\label{sec:pre}
This section provides relevant knowledge in multi-agent coverage control, the extension of performance functions to surfaces, and some theoretical aspects of surface parameterization.

\subsection{Coverage Performance Function}
For convex region $V$ embedded in $\mathbb{R}^{2}$, let $I_n=\left \{ {1,2,...,n} \right \} $ denote the set of $n$ agents, and $P=\left \{ p_{1} ,p_{2},...,p_{n}   \right \} $ represents their positions. The Voronoi cells generated by $p_i$ is given by
\begin{equation}\label{voro}
V_{i} =\left \{ q\in V|\left \| q-p_{i}  \right \|\le \left \| q-p_{j}  \right \| ,i\ne j,1\le j\le n  \right \} ,
\end{equation}
where $\left \|q-p_{i} \right \| $ represents the Euclidean distance between $q$ and $p_i$.
To achieve centroidal Voronoi tessellation for optimizing coverage performance, the following control law is constructed \cite{cortes2004coverage}.
\begin{equation}\label{control}
u_{i} =-k_{i} (p_i-C_{V_i}),
\end{equation}
% from the above equation, it can be concluded that $\dot { H} (P)\le 0$ holds when and only when $p_i=C_{V_i}$, and from Lassalle's principle the system can converge to the desired position.
where $C_{V_i}$ refers to the centroid of $V_i$.
For surfaces in 3D space, Euclidean distance is no longer applicable. Furthermore, due to the presence of slopes, the coverage cost of agents depends on multiple factors. Design the cost function for surface coverage as follows:
\begin{equation}\label{index}
H(P) =\sum_{i=1}^{n} \int_{V_i}^{} \lambda (d_{\omega } ,\zeta_{h}  )\phi  (q)dq,
\end{equation}
where $ \lambda (d_{\omega } ,\zeta_{h})$  characterizes the effect of distance and environment information on the cost. $d_{\omega (q,p_i)}$ denotes the geodetic distance from $p_i$ to $q$, and $\zeta_{h}$ represents the maximum height difference from $p_i$ to $q$, which characterizes the steepness of the surface slope, as shown in Fig.~\ref{slop}. In addition, $\phi(q)$ quantifies the probability distribution of random events.
Under the above formulation, it is a challenging task to determine the optimal deployment of agents due to the difficulty in identifying subtle changes on surfaces~\cite{wang2006brain}. Therefore, it is crucial to take the transformation from the 3D surface to the 2D plane.

\subsection{Mapping construction}
\begin{defn}
(Conformal Mapping \cite{schinzinger2012conformal}) Suppose that $Q$ is a zero-deficit surface embedded in $\mathbb{R}^{3}$ with boundary $\partial Q$ and that there exists a harmonic map $\varphi :(Q,g)\to \mathbb{D}\in R^{2} $, which transforms the surface $Q$ to a unit disk $\mathbb{D}$. Q is said to be conformal to $\mathbb{D}$ if there exists a positive scalar function $\lambda $  such that $\varphi d_{s _{Q} }  ^{2}=\lambda d_{s _{D} } ^{2}$. The terms $d_{s _{Q}}$ and $d_{s _{D} }$ represent the manifold metrics of $Q$ and $\mathbb{D}$, respectively.
\end{defn}

\begin{figure}[t!]
\centering
\includegraphics[width=3.5in]{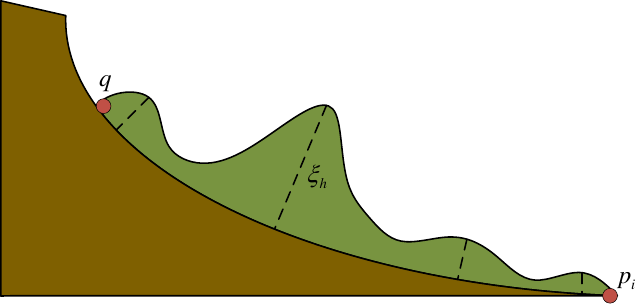}
\caption{\label{slop} Illustration on surface environments.}
\end{figure}
\begin{remark}
For a given Riemannian metric tensor, the harmonic energy of function $\varphi :Q\to \mathbb{D} $ can be defined as
\begin{equation*}
\label{deqn_ex6a}
E(\varphi )=\int_{Q}^{} \left | \bigtriangledown \varphi  \right | ^{2} d_{v_{Q} } .
\end{equation*}
This function can be employed to assess the smoothness of the mapping $\varphi $ with a critical point for harmonic mappings. 
For surfaces with zero genus, the conformal mapping is equivalent to the harmonic mapping. The disk parameterization, which constructs the harmonic mapping, can be obtained by solving the critical points of harmonic function. The construction of Laplace equation is given as follows:
\begin{align}
\begin{split}
\left \{
\begin{array}{ll}%ll按顺序是公式左对齐和条件左对齐
    \bigtriangleup _{Q} \varphi(u)=0  \: \: \:\:\: if\:\: u\in Q\setminus \partial Q \\
    \varphi |_{\partial Q}=g .\\
\end{array}
\right.
\end{split}
\end{align}
After discretizing the environment, the set of points $\left\{v_i \right\} _{i=0}^{n-1}$ on the boundary can be mapped onto the unit circle $\mathbb{D}$  in proportion to the edge length:
\begin{equation*}\label{deqn_ex8a}
\varphi(v_i)=(\cos\theta_i ,\sin\theta _i).
\end{equation*}
Denote the side length of $\left[v_i,v_{i+1}\right] $ by $l_{\left [v_i,v_{i+1}\right]}$ and solve for $\theta_i$ according to the following equation
\begin{equation*}
\left\{
\begin{aligned}
& r:=\sum_{i=0}^{n-1}l_{\left[v_{i},v_{i+1}\right]} \\
& r_i:=\sum_{j=0}^{i-1}l_{\left[v_{j},v_{j+1}\right]} \\
& \theta_i:=2\pi\frac{r_i}{r}
\end{aligned}
\right.
\end{equation*}
where $r$ is the total edge length and $r_i$ is the edge length from the start point to the point $i$.
\end{remark}

Note that directly solving the above functions may induce conformal distortions due to boundary constraints, and in the next section we will discuss how to ameliorate such distortions.

\section{Main Results}\label{sec:hol}
This section includes rectifying distortions introduced by the initial mapping, applying the conformal mapping algorithm for post-correction, and addressing the transformations to a 2D plane.
\subsection{Distortion Correction}

Due to the constraints imposed by boundary conditions in the mapping, the inevitable conformal distortions arise, which leads to angular distortions. Such distortions can cause inaccuracies in shape and relative positions, which results in a significant decrease of overall coverage performance. Therefore, it is crucial to correct this conformal distortion.
Thus,  the conformal mapping is introduced as a strategy to reduce conformal distortion, and it is a unique type of mapping with properties between holomorphic mapping and general differentiable mapping. Intuitively, the conformal mapping can transform an infinitesimal circle into a bounded eccentric ellipse. Mathematically, the conformal mapping $f:\mathbb{C}\to \mathbb{C}$ 
satisfies the Beltrami equation below.
\begin{equation*}\label{deqn_ex10a}
\frac{\partial f}{\partial\bar{z}  }=\mu (z)\frac{\partial f}{\partial z}  ,
\end{equation*}
where $\mu(z)$ refers to the Beltrami coefficient of $f$, which measures the extent of conformal distortion, and $\bar{z}$ represents the conjugate of $z$.
As depicted in Fig.~\ref{contra}, the maximum enlargement $\arg(\mu (q))/2$, the minimum contraction angles $[\arg(\mu (q))-\pi ]/2$, and the scaling factor $(1\pm \left |\mu_{q}\right|)$ can be determined, and the maximum dilation of $f$ is given by
\begin{equation*}
\label{deqn_ex11a}
K(f)=\frac{1+\left\|\mu\right \|_{\infty}}{1-\left\|\mu \right\|_{\infty}}.
\end{equation*}
It is worth mentioning that the conformal mapping will degenerate to an all-pure mapping when $ K = 1$. 
In addition, the Beltrami coefficient reflects the bijection of the mapping, as explained below.
\begin{figure}[t!]
\centering
\includegraphics[height = 3.2cm,width=3.4in]{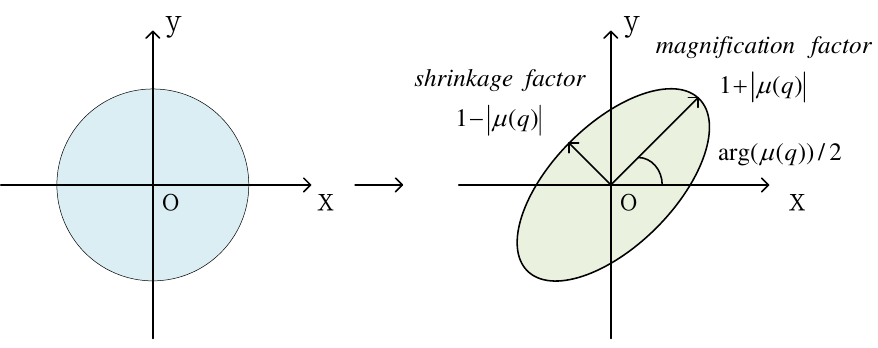}
\caption{\label{contra} Schematic diagram reflecting the degree of contraction.}
\end{figure}

\begin{lemma}
(Jacobian~\cite{choi2015fast}) For the conformal mapping $f$, the Jacobi matrix of $f$ is $J_{f}=\left | \frac{\partial f}{\partial z}  \right | ^{2} (1-\left | \mu _{f}  \right | ^{2} ) $, due to $\left \| \mu _{f}  \right \| _{\infty } < 1$, one obtains $\frac{\partial f}{\partial z} ^{2} \ne 0$ and $(1-\mu _{f} ^{2} )> 0$. Therefore, it follows that the Jacobian $J_f$ is positive everywhere.
\end{lemma}

Considering the process of constrcuting conformal mappings, it can be reconstructed to correct conformal aberrations by the Beltrami coefficients, as explained below.
\begin{lemma}\label{lem:comp}
(Composite Mapping \cite{choi2021efficient}) For the conformal mappings $f: M_{1} \to M_{2} $ and $\varphi: M_{2} \to M_{3} $, if $\mu _{f^{-1} } $ is equal to $\mu _{\varphi  }$, the Beltrami coefficient of $\varphi\circ f$ is equal to $0$, i.e., $\varphi\circ f$  is a conformal mapping.
\end{lemma}
We previously mentioned that Beltrami coefficients $\mu (p)$ are used to measure the conformal distortion of maps.
\begin{lemma}
(Distortion Evaluation \cite{halier2000nondistorting}) The mapping $f$ is conformal in a small neighborhood around point $p$ if and only if $\mu (p)=0$.
\end{lemma}

Building upon the above results, a one-to-one conformal mapping $g:\mathbb {D}\to \mathbb {D}$ can be introduced to rectify the conformal distortion of $\varphi$ that we previously constructed. Moreover, if the Beltrami coefficient of $\varphi$ is $\mu _{ \varphi }$, and another conformal mapping $g:\mathbb D\to \mathbb D$ with the same Beltrami coefficient is computed, their composite mapping is conformal according to Lemma~\ref{lem:comp}, effectively correcting the distortion. It is intuitive to reframe the issue for the conformal mapping $g$.
\begin{remark}
The linear Beltrami solver LBS is proposed to quickly solve the corresponding conformal factorization sum based on the Beltrami coefficients~\cite{lui2013texture}. For simplicity, we denote by LBS($\mu$) the conformal mapping associated with $\mu$.
\end{remark}
It becomes a nonlinear problem to find an optimal boundary $g|_{\partial \mathbb{D}}$ for minimizing conformal distortion. To simplify computation, we adopt the strategy of mapping the unit disk to the upper half-plane using the Cayley transform \cite{havu2007cayley}. Mathematically, the Cayley transform $Y:\mathbb{D}\to \mathbb{H}=\left \{ x+iy|y\ge 0;x,y\in R \right \} $  is defined as:
\begin{equation}\label{deqn_ex13a}
Y(z)=\frac{i+iz}{i-iz}.
\end{equation}
The problem then transforms into solving a conformal mapping $f:\mathbb{H}\to \mathbb{H}$ with Beltrami coefficient $\mu _{(Y\circ \varphi )}$. According to Lemma~\ref{lem:comp}, the composite mapping $f\circ Y\circ \varphi$ is conformal. It's worth noting that under the Cayley transform, the boundary of the disk is mapped to the real axis $y = 0$.
After discretizing the environment with a triangular mesh, the initial mapping $\varphi$ maps $Q$ to the triangular mesh $\Omega$ of the disk $\mathbb{D}$, and $Y$ maps $\mathbb{D}$ to a large triangle with vertices $\left[ p_{1}, p_{2}, p_{3} \right]$. Next, we solve the system of equations with these three points.
\begin{equation}
\left\{
\begin{aligned}
&f=LBS(\mu_{ (Y\circ \varphi )^{-1}} )   \\
&f(Y(p_{i} ))=Y(p_{i} )\:\:\:\:\:i=1,2,3  \\
&Im(f(Y(z)))=0\:\:\:\: z\in \partial \mathbb{D}.
\end{aligned}
\right.
\end{equation}
After solving the above system of equations, we map the points in the upper half-plane back to the unit disk using the inverse of the Cayley mapping $Y^{-1}(z)=\frac{z-i}{z+i}$. Finally, by integrating these mappings, we obtain a conformal mapping $g:=Y^{-1} \circ f\circ Y\circ \varphi$.
\subsection{Algorithm for Disk Conformal Mapping}
This subsection describes the implementation of disk conformal mapping algorithm.
The first step is the construction of the initial mapping, which involves discretizing the environment $Q$ into a triangular lattice $K$. Let $\left[u, v\right]$ denote the edges connecting the two vertices $u$ and $v$, at which point the potential energy of $\varphi :Q\to \mathbb{D}$ is given by
\begin{equation*}
\label{deqn_ex16a}
E(\varphi )=\sum_{[u,v]\in K} k_{uv} \left \|\varphi (u)-\varphi(v)\right\|^{2},
\end{equation*}
with $k_{uv}=\cot\alpha +\cot\beta$. Here $\alpha$ and $\beta$ are the angles opposite side $[u,v]$, as shown in Fig.~\ref{fig:wei}. The Laplace operator can be discretized as
\begin{equation}
\label{deqn_ex17a}
\Delta _{Q} \varphi (v_{i})=\sum_{v_{j}\in N(V_{i} )} k_{v_{i}v_{j}} [\varphi (v_{j} )-\varphi (v_i)],
\end{equation}
where ${v_{j}\in N(v_{i})=\left \{ v_j|k_{v_iv_j}\ne 0 \right \}}$, $N(v_{i})$ is the set of neighbors of vertex $v_i$, thus transforming (4) into a sparse linear system.
\begin{figure}[t!]
\centering
\includegraphics[width=0.5\linewidth]{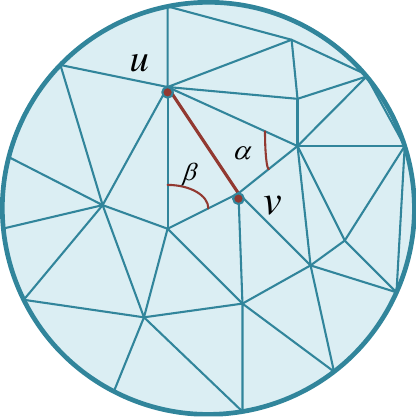}
\caption{\label{fig:wei} Schematic representation of weighting factors $k_{uv}$.}
\end{figure}
Let $\varphi =(u+iv):T_{1} \to T_{2}$ be a directionally preserving homeomorphic mapping between two planar triangular meshes. To calculate the Beltrami coefficient of $\varphi$, we approximate the partial derivative on each triangular surface $S_1$ on $T_1$. Suppose that the face $S_1$ on $T_1$ is transformed by mapping to $S_1$ on $T_2$. The approximation of $\mu _{\varphi } $ on $S_1$ can be computed using the coordinates of six vertices on these two faces. By setting $S_1=[a_1+ib_1,a_2+ib_2,a_3+ib_3]$ and $S_2=[w_1,w_2,w_3]$, we approximate the Beltrami coefficients on $S_1$ as
\begin{equation*}\label{deqn_ex18a}
\mu _{\varphi (S_{1} )}=\frac{\frac{1}{2} (F_{x}+iF_{y} )\begin{pmatrix}w_1
 \\w_2
 \\w_3\\
\end{pmatrix} }{\frac{1}{2} (F_{x}-iF_{y} )\begin{pmatrix}w_1
 \\w_2
 \\w_3\\
\end{pmatrix}}
\end{equation*}
with
%\begin{small}
\begin{equation*}\label{deqn_ex19a}
F_{x} =\frac{1}{2Area(S_1)} \begin{pmatrix}b_3-b_2
 \\b_1-b_3
 \\b_2-b_1
\end{pmatrix}^{T}, \quad F_{y} =\frac{1}{2Area(S_1)} \begin{pmatrix}a_3-a_2
 \\a_1-a_3
 \\a_2-a_1
\end{pmatrix}^{T}.
\end{equation*}
%\end{small}
The process of discretization using LBS is briefly outlined below. Let $S=[v_i,v_j,v_k]$ represent a triangular face, where $v_i$ denotes vertex $i$, and $j,k\in N_{i}$, with $N_{i}$ being the set of neighbors of vertex $i$. Let $u_t=\varphi(v_t)$ for $t=i,j,k$. By assuming $v_{j} =g_{j} +ih_{j}$ and $u_{j} =s_{j} +it_{j}$, and expressing the Beltrami coefficients of $S$ as $\mu _{\varphi }(S) =\rho _{S} +i\eta _{S}$, the LBS obtains the discrete linear equation. The desired mapping function can be obtained by solving these linear equations. The detailed implementation is presented in Algorithm~\ref{alg:alg1}.

\begin{algorithm}[t!]
\caption{\label{alg:alg1} Conformal Mapping Construction Algorithm}
\hspace*{0.02in} {\bf Input:} The triangularized lattice $K$ of $Q$  \\ %\ENSURE
\hspace*{0.02in} {\bf Output:} $g: K\to\mathbb{D}$
\begin{algorithmic}[1]     %\REQUIRE     %算法的输入参数：Input
\State \textbf{for} i= 1 : N \textbf{do}
       \State Calculate $l_{[v_i, v_{i+1}]}$
       \State Update $\varphi:K\to\mathbb{D}$ with (4)
\State \textbf{end}
\State Linearization by (5)
\State Solve $f=LBS(\mu_{(Y\circ\varphi)^{-1}})$ with $\mu_{(Y\circ\varphi)^{-1}}$
\State Map to a disk by $Y^{-1}$
\State Finalize $g:=Y^{-1} \circ f\circ Y\circ\varphi$
\end{algorithmic}
\end{algorithm}

\subsection{Coverage Optimization}
%After applying the disk conformal mapping algorithm to transform the environment into a 2D disk, the subsequent section outlines the approach for covering the transformed environment.
Let $P=[p_{1},p_2,...,p_n ]$ denote the initial position of MAS in the region $Q$, and let $f:Q\to \mathbb{D}$ be the proposed harmonic mapping. Based on the position of MAS using the partition rule in Eq.~(\ref{voro}) for $\mathbb{D}$, for each subregion $V_i$, the center of mass $C_{V_i}$ and mass $M_{V_i}$ are given by
\begin{equation*}\label{deqn_ex20a}
 C_{V_{i} } =\frac{\int_{V_i}^{} q\phi   (q)dq}{\int_{V_i}^{} \phi   (q)dq} ,\:\:\: M_{V_{i} } ={\int_{V_i}^{} \phi  (q)dq} .
\end{equation*}
In accordance with Eq.~(\ref{control}), we derive the control law for guiding agents toward the center of mass within the convex region. On the mapped 2D disk, $\lambda (d_{\omega } ,\zeta_{h}  )$ in (\ref{index}) becomes $\left \| f(q)-f(p_i) \right \|^{2}$.
\begin{figure*}[t!]
\centering
\subfloat[\label{fig:a}Original surface]{\includegraphics[width=3.4in]{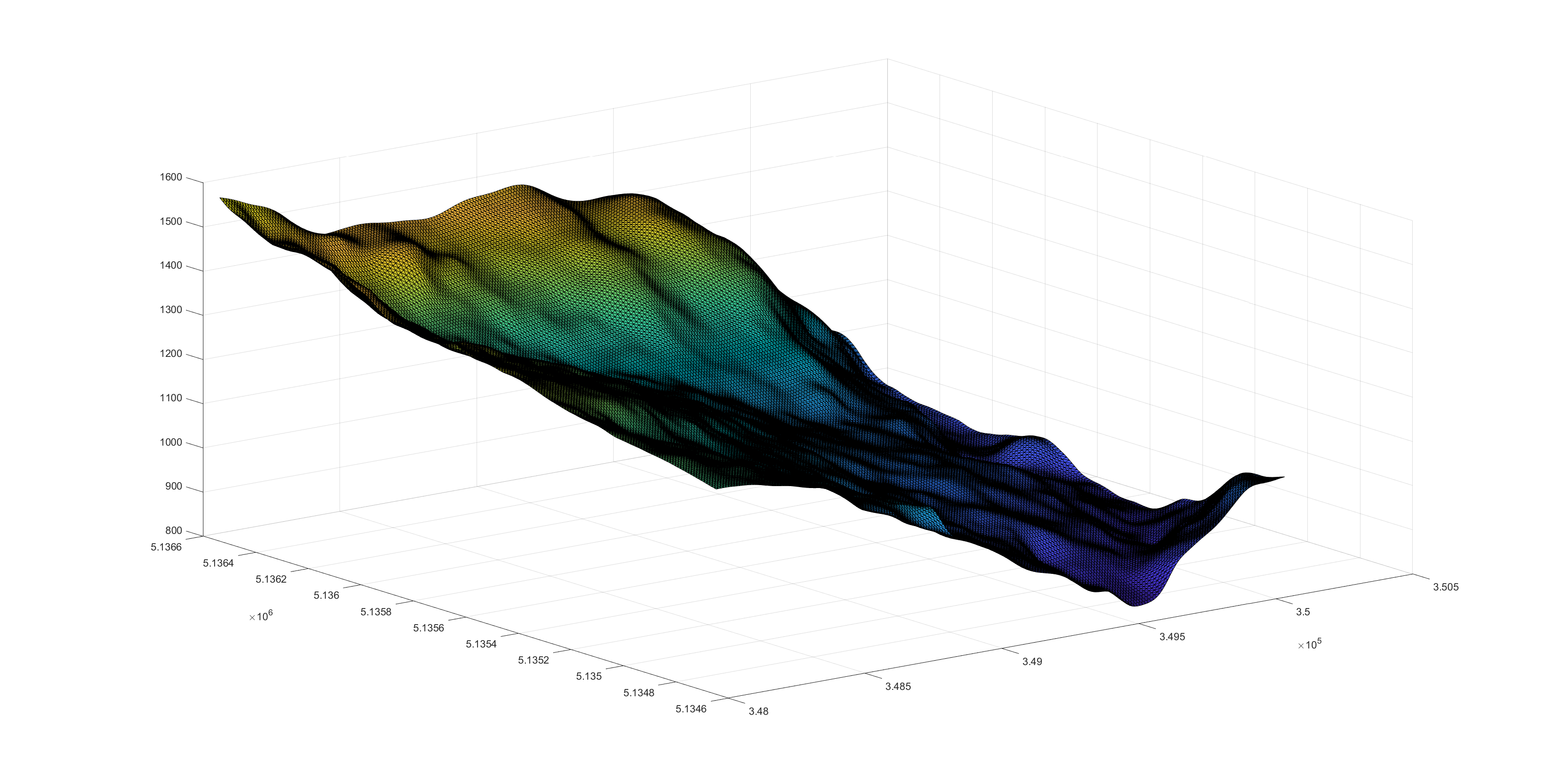}}
\hfil
\subfloat[\label{fig:b}Deformed surface]{\includegraphics[width=3.4in]{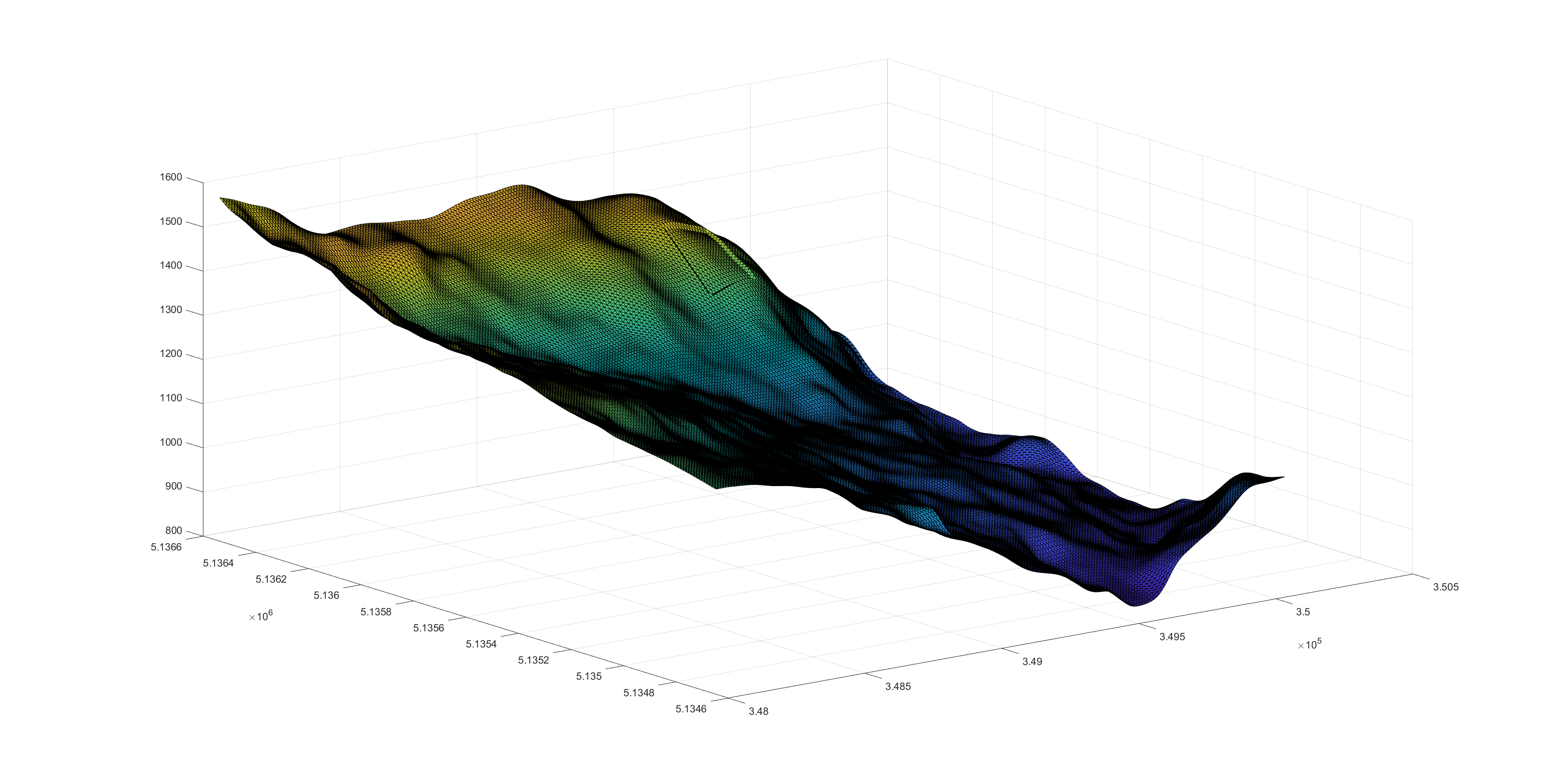}}
\hfil
\subfloat[\label{fig:a}Initial mapping disc]{\includegraphics[width=3.4in]{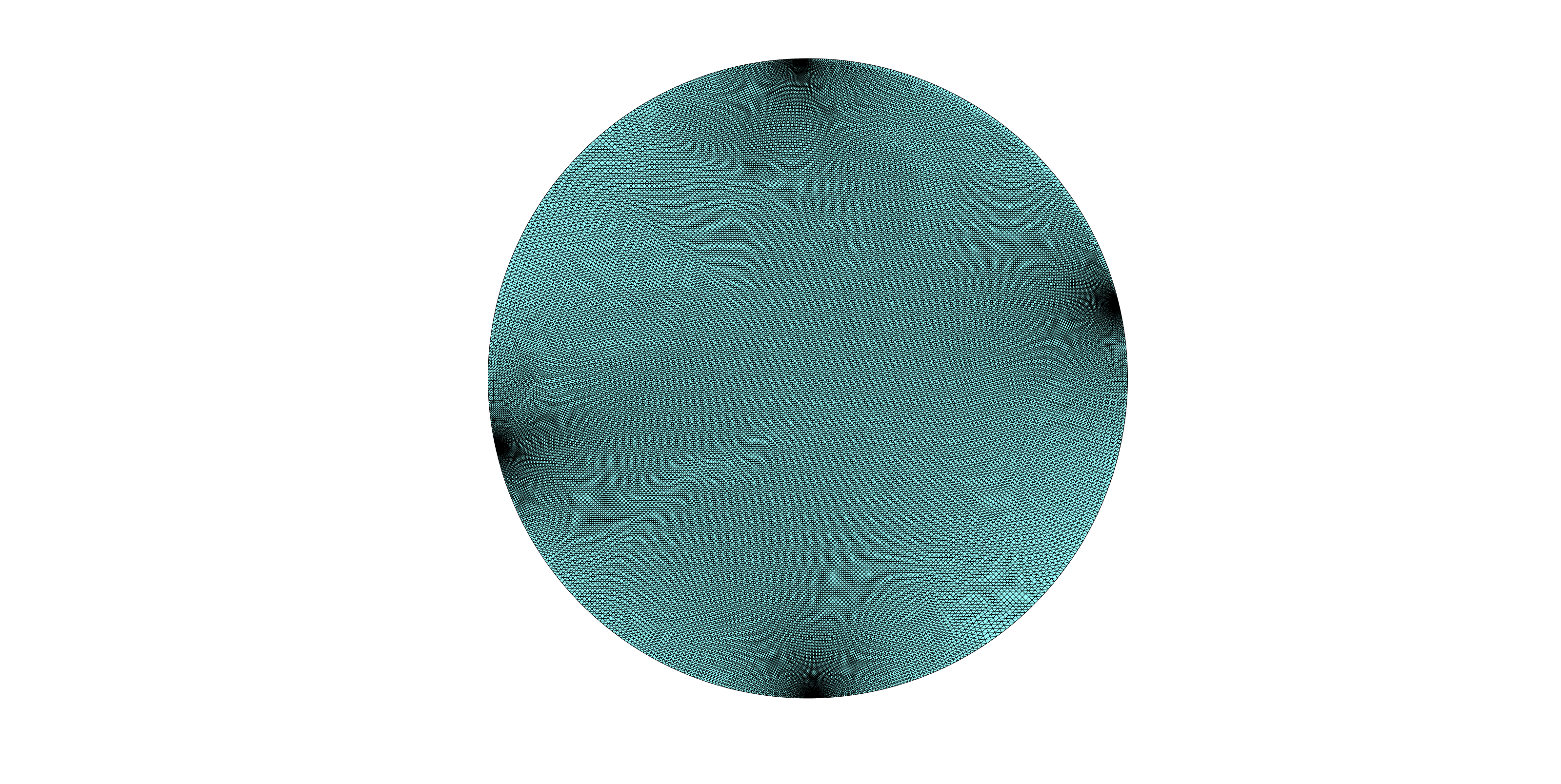}}
\hfil
\subfloat[\label{fig:b}Mapping disc after deformation]{\includegraphics[width=3.4in]{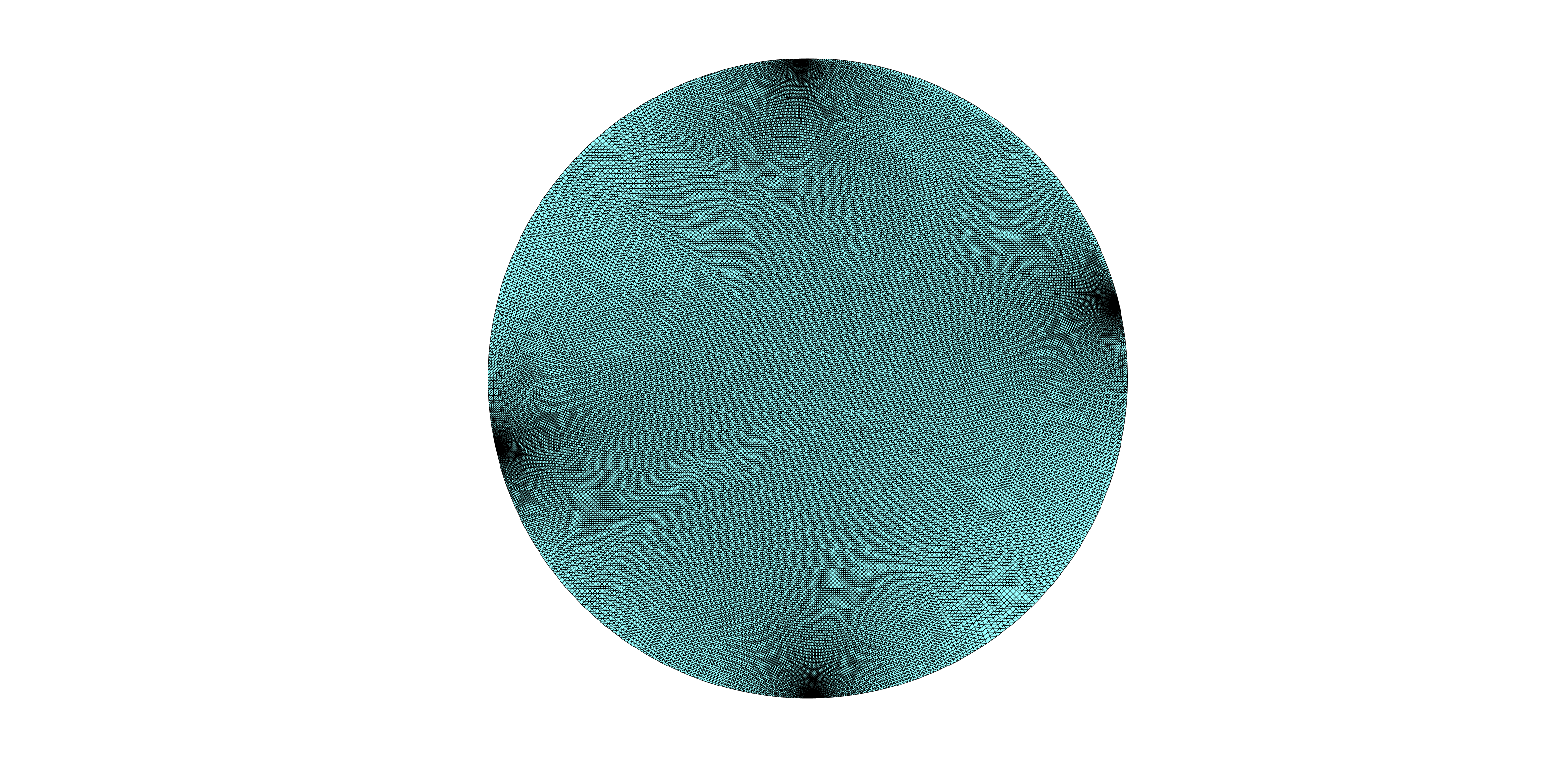}}
\caption{\label{fig:mul}Comparison between original and deformed surfaces.}
\label{fig_sim}
\end{figure*}
\begin{theorem}
The control law
\begin{equation}\label{u_con}
u_i =-k_iJ_{f}^{-1} (f(p_i))(p_i-f^{-1}(C_{V_i}) )
\end{equation}
solves the coverage problem
\begin{equation}\label{min}
\int_{Q}\min_{i\in I_n}\left \| f(q)-f(p_i) \right \|^{2}  \phi  (q)dq,
\end{equation}
where $J_f=\frac{\partial f}{\partial q} $ represents the Jacobian of $f$. 
\end{theorem}
\begin{proof}
By the transformation of variables, if $\Omega \subset R^n$ is an open set, $g:\Omega \to \mathbb{R}$ is integrable over $\Omega$, and if $\psi :\tilde{\Omega } \to \Omega $ is a differential homomorphism, then $(g\circ \psi )\left | detJ_\psi  \right | $ is integrable over $\tilde{\Omega }$, and the following equation holds
\begin{equation*}\label{deqn_ex23a}
\int_{\Omega }^{} g(r)dr=\int_{\tilde{\Omega }}^{}  (g\circ \psi )(\tilde{r} )\left | detJ_{\psi } (\tilde{r} ) \right |d\tilde{r},
\end{equation*}
where $| detJ_{\psi } (\tilde{r} ) |$ is the determinant of the Jacobi matrix of $\psi$ computed on point $\tilde{r}$.
Since $Q$ is compact and the integral in (\ref{min}) is a composition of continuous functions that are integrable over $Q$, we get with $\varsigma_i=f(p_i)$
\begin{equation}\label{trans}
\begin{split}
&~~~~\int_{\mathbb{D}}\min_{i\in I_n}\left\| w-\varsigma_{i} \right\| ^2\phi (f^{-1}(w))dw\\
&=\int_{Q}\min_{i\in I_n}\left \| f(q)-f(p_i) \right\|^2\phi(q)dq,
\end{split}
\end{equation}
where $w=f(q)$, $q\in Q$. Thus, we derive the control law for the suboptimal deployment on $\mathbb{D}$
\begin{equation*}\label{deqn_ex24a}
 u_i =-k_i(f(p_i)-C_{V_i}),
\end{equation*}
where $C_{V_i}$ represents the center of mass of the Voronoi cell induced by $f(p_i)$. Since $f$ is bijective, the path generated by the control law (\ref{trans}) has a unique counterpart in the original region $Q$ with a smooth preimage. This completes the proof.
\end{proof}

\begin{coro}\label{cor1}
If the mapping $f:V\to \mathbb{D}$ is conformal and bijective, the continuous Voronoi cells in $\mathbb{D}$ map back to the original environment as a continuous set.
\end{coro}
\begin{proof}
Let $V_i\in \mathbb{D}$ be an open set, and we need to prove that $f^{-1}(V_i)$ is also an open set.
Since $f$ is a bijection, for any $q\in V_i$, there exists a unique $v\in V,f(v)=q$. For each $q$ in $V_i$, there is a unique $v$ in $f^{-1}(q)$ corresponding to it,
\begin{equation*}\label{deqn_ex24a}
  f^{-1}(V_i)=\cup \left \{ v\in V|f(v)\in V_i \right \} ,
\end{equation*}
as $V_i$ is an open set and $f^{-1}(V_i)$ is the union of some open sets in $V$, $f^{-1}(V_i)$ is thus an open set in $V$.
\end{proof}

\begin{coro}\label{cor2}
If the Voronoi cells in $\mathbb{D}$ have no intersections, then the mapped Voronoi partition back to the original space is also non-intersecting.
\end{coro}
\begin{proof}
The conformal and angle-preserving nature of the mapping ensures that Voronoi cells retain their shapes and angles throughout the inverse mapping process \cite{kythe2012computational}, effectively preventing any intersections.
\end{proof}
By partitioning the region in the mapped disk to design the control law for coverage optimization, and subsequently mapping the resulting deployments and paths back to the original surface, this strategy facilitates an effective reduction in the performance function, thereby achieving a suboptimal deployment of MAS.

\section{Simulations}\label{sec:sim}
This section provides simulation results of coverage control algorithm based on conformal mapping, accompanied by a comparison with traditional 3D Voronoi partition.
\begin{figure}[t!]
\centering
\includegraphics[width=3in]{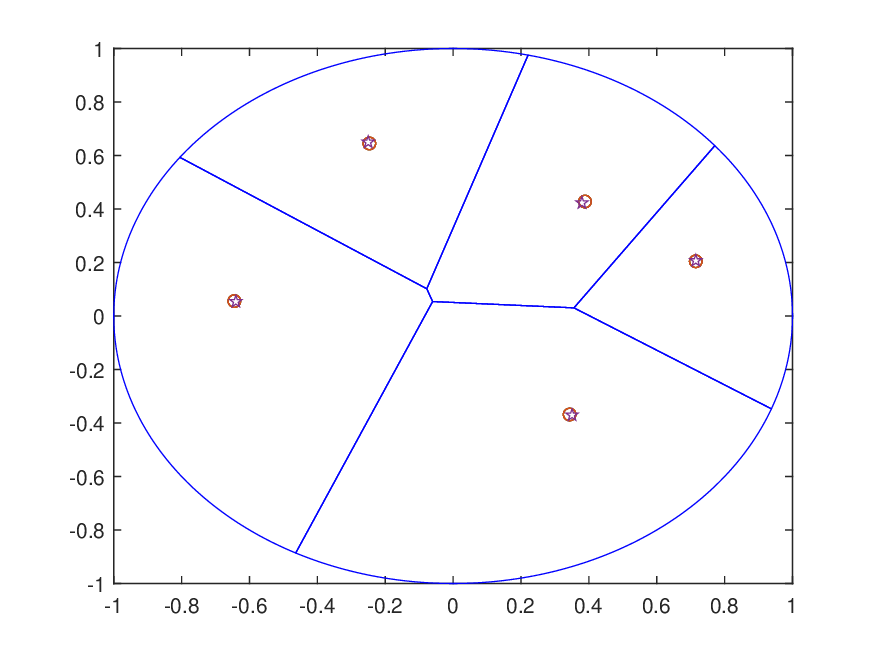}
\caption{\label{fig:part}Multi-agent coverage on the disk.}
\end{figure}
Initially, environmental modeling is carried out by integrating various data, leading to the creation of a surface environment as depicted in Fig.~\ref{fig:mul}(a), and the deformed surface image is depicted in Fig.~\ref{fig:mul}(b). However, in real-world scenarios, especially in the monitoring of surface deformations, the initial deformation of the surface is extremely subtle and difficult to discern in 3D space.
\begin{figure*}[t!]
\centering
\subfloat[\label{fig:a}Method based on conformal mapping]{\includegraphics[width=2.8in]{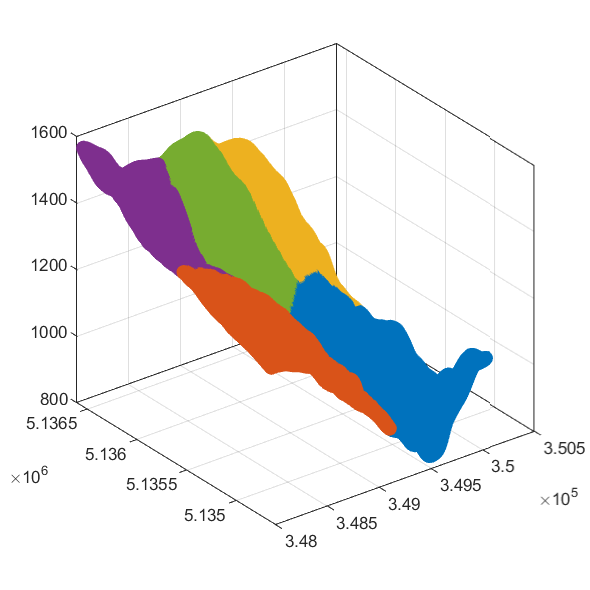}}
\hfil
\subfloat[\label{fig:b}Traditional Voronoi partitioning method]{\includegraphics[width=2.8in]{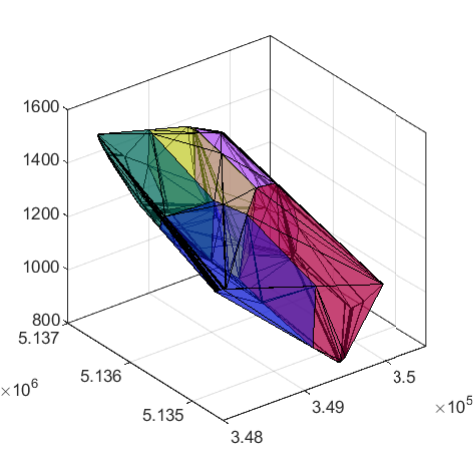}}
\caption{\label{fig_sim} Comparison of surface partitioning diagrams.}
\end{figure*}
Figure~\ref{fig:mul} also depicts simulation results of mapping a 3D surface to a 2D disk using the constructed conformal mapping. Panels (a) and (b) showcase images of the disk with triangular mesh that represents the original and deformed states, respectively.
The deformation characteristics of surface are observable from the disk image with a triangular section. Subsequently, the discs before and after deformation are compared to identify the deformation components. The coverage control algorithm is then applied to the disk, and the final deployment of agents and region partition are illustrated in Fig.~\ref{fig:part}, where the pentagrams denote the centroid of Voronoi cells, and the hollow circles represent the agent positions from the previous iteration.
Thanks to the bijective nature of the mapping, we can map the Voronoi cells on the disk back to the original surface. This eliminates the necessity for partitioning the original environment. The final region partition is depicted in Fig.~\ref{fig_sim}(a). It is evident that the partition is continuous and non-intersecting, as indicated in Corollaries~\ref{cor1} and \ref{cor2}. Applying the conventional 3D Voronoi partition directly to surfaces poses challenges when using the same environmental data~\cite{gou2023vkece}. Traditional Voronoi partition relies on convex regions and necessitates generating the convex hull of the surface prior to partitioning. However, this approach may result in loss of original surface characteristics, as depicted in Fig.~\ref{fig_sim}(b). Although refining the partition by iteratively determining whether points on the surface lie within the convex partition subregion may be feasible, such a process is intricate. Furthermore, it may lead to discontinuous partitioning in certain extreme cases.
Observing the comparison in Fig.~\ref{fig_sim}, it can be seen that, with the same environmental data, the partition method based on conformal mapping can divide the environment into finer contiguous subregions. In contrast, traditional 3D Voronoi partition fails to recognize surface characteristics. Given the bijective nature of conformal mapping, both the final deployment and the path of agents induced by the control laws on the 2D disk cna be mapped back to the original surface, as depicted in Fig.~\ref{path}. During the execution of algorithm, time evolution of coverage performance function is illustrated in Fig.~\ref{perf}. 
It is clear that multi-agent coverage algorithm effectively detects the real-time status of the concerned surface. 
Moreover, it optimizes coverage performance based on surface precursor features, such as deformation metrics.

\begin{figure}[t!]
\centering
\includegraphics[width=3.5in]{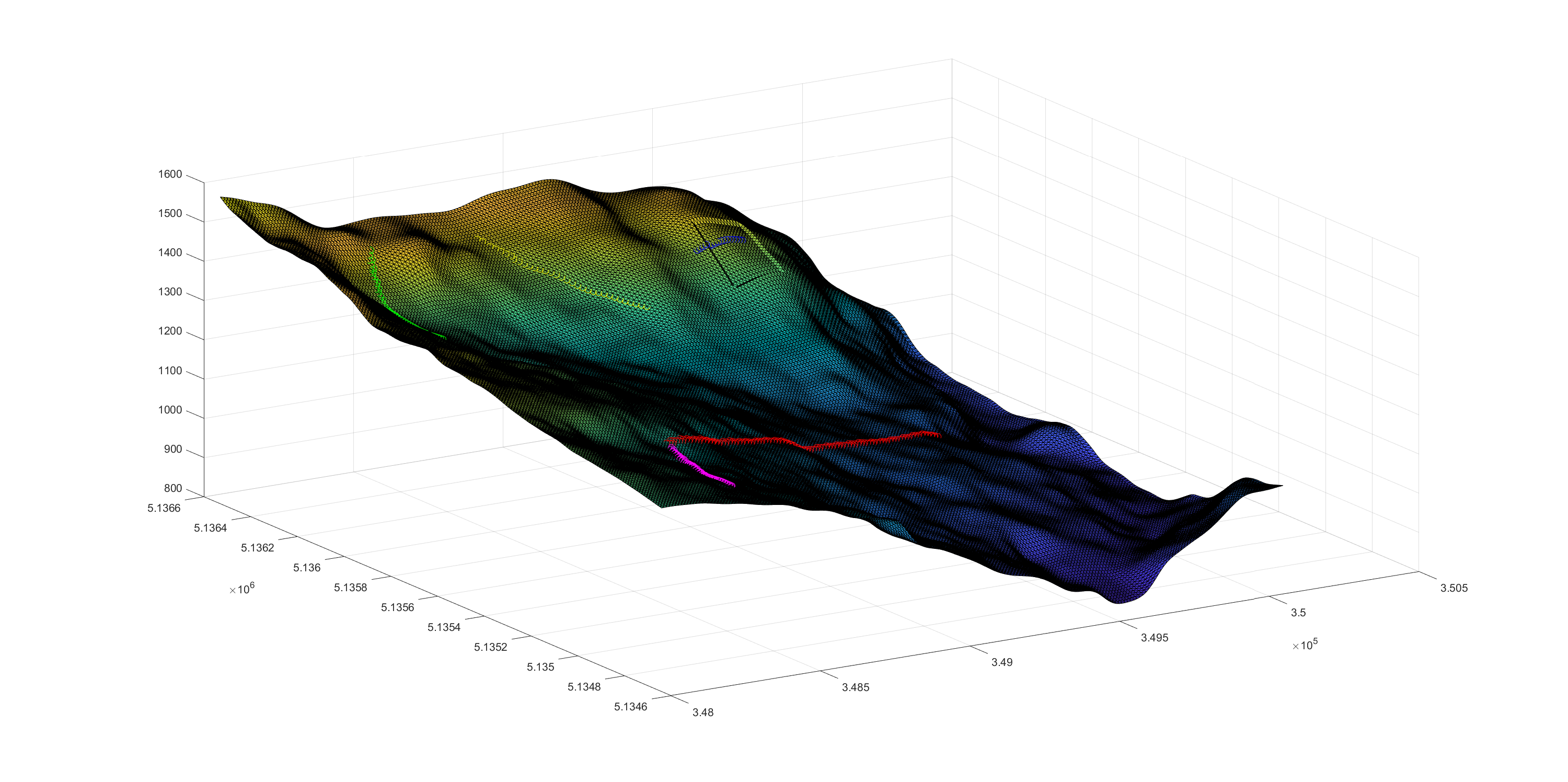}
\caption{\label{path} Path planing on the original surface.}
\end{figure}

\begin{figure}[t!]
\centering
\includegraphics[width=3in]{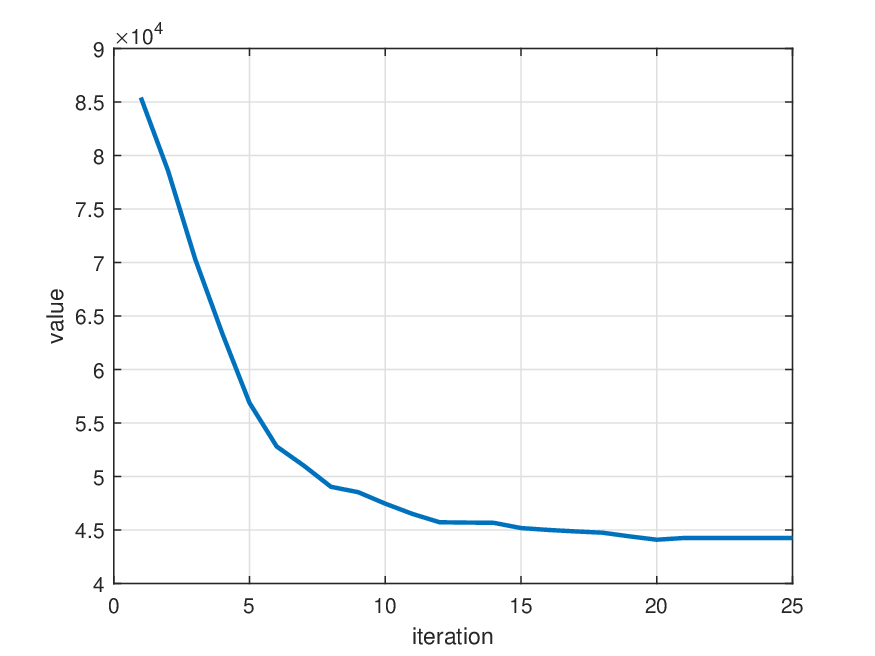}
\caption{\label{perf} Time evolution of coverage performance function.}
\end{figure}

%\begin{figure}[t!]
%\centering
%\includegraphics[width=3in]{distort.eps}
%\caption{The norm histogram of Beltrami differences.}\label{fig_9}
%\end{figure}

\section{Conclusions}\label{sec:con}
In this paper, we introduced a surface monitoring approach based on 3D spatial modeling. The method utilized environmental transformation to achieve comprehensive coverage. Through simulation experiments, the algorithm demonstrated robust coverage performance with computational efficiency. The process involved mapping and processing the environment, partitioning and discovering optimized deployments in the mapped disc based on surface shape variables, and finally mapping back the optimized deployments and paths of 2D space to the original environment. Furthermore, this method facilitated the observation of surface deformation characteristics to overcome certain limitations in 3D space. This method have promising potential for applications 
in addressing challenges related to feature extraction of geohazards. The versatility of this method extends beyond general surfaces to practical applications such as landslide monitoring. Future work will focus on utilizing this mapping algorithm for extracting precursor features of geohazards and on designing more realistic metrics to optimize the coverage performance.

\bibliographystyle{IEEEtranN}
%\balance
\bibliography{wx}

\vfill

\end{document}